\documentclass[11pt]{amsart}
\usepackage{latexsym}
\usepackage{amsmath,amssymb,amsfonts}
\usepackage{graphics}
\usepackage[all]{xy}

\newtheorem{thm}{Theorem}
\newtheorem{lemma}{Lemma}[section]

\newtheorem{remark}{Remark}

\newenvironment{pf1}{\mbox{ }\\{\bf  Proof of Theorem 1}\mbox{ }}{
\hfill $\Box$\mbox{}\bigskip}
\newenvironment{pf2}{\mbox{ }\\{\bf  Proof of Theorem 2}\mbox{ }}{
\hfill $\Box$\mbox{}\bigskip}
\newenvironment{pf3}{\mbox{ }\\{\bf  Proof of Theorem 3}\mbox{ }}{
\hfill $\Box$\mbox{}\bigskip}
\newenvironment{pf4}{\mbox{ }\\{\bf  Proof of Theorem 4}\mbox{ }}{
\hfill $\Box$\mbox{}\bigskip}

\title{$\textbf{1}$. Introduction}

\begin{document}


\title[Three-dimensional Ricci-degenerate Riemannian manifolds]{Three-dimensional Ricci-degenerate Riemannian manifolds satisfying geometric equations}

\author{Jinwoo Shin}
 \maketitle

\begin{abstract}

In this paper, we study a three-dimensional Ricci-degenerate Riemannian manifold $(M^3,g)$  that admits a smooth nonzero solution $f$ to the equation
\begin{align} \label{a1a}
  \nabla df=\psi Rc+\phi g,
\end{align}
where $\psi,\phi$ are given smooth functions of $f$, $Rc$ is the Ricci tensor of $g$. Spaces of this type include various interesting classes, namely gradient Ricci solitons, $m$-quasi Einstein metrics, (vacuum) static spaces, $V$-static spaces, and critical point metrics.

 The $m$-quasi Einstein metrics and vacuum static spaces were previously studied in \cite{JJ,JEK}, respectively. In this paper, we refine them and develop a general approach for the solutions of (\ref{a1a}); we specify the shape of the metric $g$ satisfying (\ref{a1a}) when $\nabla f$ is not a Ricci-eigen vector. Then we focus on the remaining three classes, namely gradient Ricci solitons, $V$-static spaces, and critical point metrics. Furthermore, we present classifications of local three-dimensional Ricci-degenerate spaces of these three classes by explicitly describing the metric $g$ and the potential function $f$.

\end{abstract}

\section{Introduction}

\renewcommand{\thepage}{\arabic{page}}

We consider a three-dimensional Riemannian manifold $(M^3,g)$ that admits a smooth nonzero solution $f$ to the equation
\begin{align}
  \nabla df=\psi Rc+\phi g, \label{ggrsdef}
\end{align}
where $\psi,\phi$ are given smooth functions of $f$, $Rc$ is the Ricci tensor of $g$. One can easily see that if $\psi\neq0$ and $f$ is a constant function (we call this {\it trivial}), then this is nothing but an Einstein manifold equation.  In this paper, we study
some well-known classes of spaces with the function $f$ that satisfies (\ref{ggrsdef}).

The first class is when $\psi=-1$ and $\phi=\lambda$, where $\lambda$ is a constant, i.e., $(M^3,g,f)$ is a \textit{gradient Ricci soliton} satisfying the following equation:
\begin{align}
  \nabla df+Rc=\lambda g. \label{soldef}
\end{align}
A gradient Ricci soliton is said to be shrinking, steady, or expanding if $\lambda$ is positive, zero, or negative, respectively.
The gradient Ricci soliton has significance as a singularity model of Ricci flow. Therefore, considerable effort has been devoted toward understanding its geometric properties and classifying it.
Studies have been conducted under various geometric conditions. Some of these studies are related to this paper, as stated below.

 In \cite{BM}, J. Berstein and T. Mettler classified two-dimensional  complete gradient Ricci solitons. In \cite{CCZ}, any three-dimensional complete noncompact non-flat shrinker was
 proved to be a quotient of the round cylinder $\mathbb{S}^2\times \mathbb{R}$; see also \cite{Iv3, NW, P}. For the three-dimensional gradient steadiers, S. Brendle showed that a three-dimensional steady gradient Ricci soliton that is non-flat and $\kappa$-noncollapsed is isometric to the Bryant soliton up to scaling in \cite{Br}; see also \cite{Cao1}.
 In higher dimensions, gradient Ricci solitons have been studied when the Weyl tensor satisfies certain conditions. Complete locally conformally flat solitons have been studied in \cite{CC1, CWZ, CM, PW2, Z}. Gradient Ricci solitons with harmonic Weyl tensors have been studied in \cite{FG, Ki, MS, WWW}. Bach-flat cases have been studied in \cite{CCCMM, CC2}.

Next, we consider a \textit{$V$-static space} \cite{CEM} that is a Riemannian manifold $(M,g)$ that admits a smooth  function $f$ satisfying
\begin{align}
  \nabla df =f(Rc-\frac{R}{n-1}g)-\frac{\kappa}{n-1}g \ \ \ \ {\rm for} \ {\rm a} \  {\rm constant} \ \kappa.\label{miaodefin}
\end{align}

Note that the existence of a nonzero solution to (\ref{miaodefin}) guarantees that the scalar curvature is constant. The (vacuum) static spaces corresponding to $\kappa =0$ have been studied in general relativity since the beginning of 20th century. Moreover, they have been the focus of a recent study \cite{QY2} related to positive mass conjecture. In addition, the geometric significance of
this class of spaces for $\kappa \neq 0$  has been well studied in \cite{MT2, MST} and especially in \cite{CEM}.
Harmonic Weyl cases were studied in \cite{ BBR, JJ2}. More related to this article, three-dimensional Ricci-degenerate spaces in the case $\kappa=0$ were already studied in  \cite{Le, JEK} by different arguments from this article.  Here we mainly study the case $\kappa\neq0$.

Finally, one may consider Riemannian metrics $(M, g)$ of constant scalar curvature that admit a non-
constant solution $f$ to
\begin{align}
  \nabla df=f(Rc-\frac{R}{n-1}g)+Rc-\frac{R}{n}g. \label{cpedef}
\end{align}
If $M$ is a closed manifold, then $g$ is a critical point of the total scalar curvature functional defined on the space of Riemannian metrics with unit volume and with constant scalar curvature on $M$. By an abuse of terminology, we shall refer to a metric $g$ satisfying (\ref{cpedef}) as \textit{a critical point metric} even when $M$ is not closed. In \cite{Be}, the conjecture that a compact critical point metric is always Einstein was
raised. This conjecture was verified under some geometric conditions, namely locally conformally flat \cite{La}, harmonic curvature \cite{HCY}, and Bach-flat \cite{QY} conditions.
A number of studies have investigated this subject, including \cite[Section 4.F]{Be} and \cite{BR,  JJ2, HY}.

In this paper, we study three-dimensional Riemannian manifolds with two distinct Ricci eigenvalues having a nonzero solution to (\ref{ggrsdef}). This paper is the result of efforts devoted toward refining and generalizing \cite{JJ}, which concerns three-dimensional $m$-quasi Einstein manifolds.
We preferentially study the common properties of spaces for general $\psi(f)$ and $\phi(f)$. We show that when $\nabla f$ is not an Ricci-eigen vector, the metric $g$ must have a specific form. Then, we focus on each of the three specific classes stated above.
We mainly adopt a local approach and state local versions of theorems below for the result.
 Other versions for complete spaces can follow readily, as the three classes of spaces are real analytic. One may also note below that these local spaces have their own geometric significance.

\begin{thm}\label{metricthm1}
   Let $(M^3,g,f)$ be a three-dimensional Riemannian manifold satisfying (\ref{ggrsdef}) with Ricci eigenvalues $\lambda_1\neq\lambda_2=\lambda_3$. Consider an orthonormal Ricci-eigen frame field $\{E_i \ | \ i=1,2,3\}$ in an open subset of $\{\nabla f\neq 0\}$ such that $E_2\perp \nabla f$ and $E_3f\neq0$. Then there exists a local coordinate system $(x_1,x_2,x_3)$ in which the metric $g$ can be written as
   \begin{align}
    g=g_{11}(x_1,x_3)dx_1^2+g_{33}(x_1,x_3)v(x_3)dx_2^2+g_{33}(x_1,x_3)dx_3^2 \label{startmetric22}
  \end{align}
  for a function $v(x_3)$ where $E_i=\frac{1}{\sqrt{g_{ii}}}\frac{\partial}{\partial x_i}$.
 \end{thm}

Note that one can always choose an orthonormal Ricci-eigen frame field $\{E_i\ | \ i=1,2,3\}$ satisfying $\nabla f\perp E_2$ without loss of generality. 

\begin{thm}
  Let $(M^3,g,f)$ be a three-dimensional gradient Ricci soliton with Ricci-eigenvalues $\lambda_1\neq\lambda_2=\lambda_3$. Then near each point in the open dense
  subset $\{\nabla f\neq 0\}$ of $M$, there exist local coordinates $(x_1,x_2,x_3)$ in which $(g,f)$ can be one of the following:

  {\rm (i)} $g=dx_1^2+h(x_1)^2\tilde{g}$ where $\tilde{g}$ has constant curvature. In particular, $g$ is conformally flat.

 {\rm (ii)} $g=dx_1^2+\tilde{g}$ where $\tilde{g}$ is a 2-dimensional gradient Ricci soliton together with a potential function $\tilde{f}$. And the potential function is $f=\frac{\lambda}{2}x_1^2+\tilde{f}$.
\end{thm}

 In \cite{CMM}, it has been shown that a complete three-dimensional Ricci-degenerate simply connected steady gradient Ricci soliton is either isometric to the Riemannian product of $\mathbb{R}$ with a surface or locally conformally flat.
 Their argument depends crucially on the completeness assumption of a soliton, which guarantees nonnegativity of sectional curvatures.
 Our argument is purely local; nonetheless, our Theorem 1 holds for expanding solitons as well as steady ones.

\begin{thm} \label{th2}
    Let $(M^3,g,f)$ be a three-dimensional $V$-static space $(\kappa\neq0)$ with Ricci-eigenvalues $\lambda_1\neq\lambda_2=\lambda_3$. Then near each point in the open dense
  subset $\{\nabla f\neq 0\}$ of $M$, there exist local coordinates $(x_1,x_2,x_3)$ in which $(g,f)$ can be one of the following:

  {\rm (i)} $g=dx_1^2+h(x_1)^2\tilde{g}$ where $\tilde{g}$ has constant curvature. In particular, $g$ is conformally flat.

 {\rm  (ii)} $g=\frac{1}{\{q(x_3)+b(x_1)\}^2}\{dx_1^2+(q')^2dx_2^2+dx_3^2\}$ where $q(x_3)$ and $b(x_1)$ satisfy $(q')^2-2mq^3-lq^2+\alpha q+k=0$ and
     $(b')^2-2mb^3+lb^2+\alpha b+\frac{R}{6}-k=0$ for constants $m\neq0$, $l,\alpha,k$, respectively. The potential function is $f=\frac{c(x_1)}{q+b}$ where $c$ is a solution to $b''c=b'c'-\frac{\kappa}{2}$.

\end{thm}

The explicit spaces in Theorem \ref{th2}  {\rm  (ii)} cannot be defined on closed manifolds, and
most of these are not even complete. However, recent studies have effectively provided a geometric meaning for even local V-static spaces; see \cite{Yu} and \cite[Theorem 2.3]{CEM}.
The spaces in Theorem 2 are neither warped products nor conformally flat.

\begin{thm} \label{th3}
  Let $(M^3,g,f)$ be a three-dimensional critical point metric with Ricci-eigenvalues $\lambda_1\neq\lambda_2=\lambda_3$. Then near each point in the open dense
  subset $\{\nabla f\neq 0\}$ of $M$, there exist local coordinates $(x_1,x_2,x_3)$ in which $(g,f)$ can be one of the following:

  {\rm (i)} $g=dx_1^2+h(x_1)^2\tilde{g}$ where $\tilde{g}$ has constant curvature. In particular, $g$ is conformally flat.

  {\rm  (ii)} $g=\frac{1}{\{q(x_3)+b(x_1)\}^2}\{dx_1^2+(q')^2dx_2^2+dx_3^2\}$ where $q$ and $b$ satisfy $(q')^2-2mq^3-lq^2+\alpha q+k=0$ and
     $(b')^2-2mb^3+lb^2+\alpha b+\frac{R}{6}-k=0$ for constants $m\neq0$, $l,\alpha,k$, respectively. The potential function is $f=\frac{c(x_1)}{q+b}-1$ where $c$ is a solution to $b''c=b'c'+\frac{R}{6}$.

 {\rm  (iii)}  $g=p^2dx_1^2+(p')^2dx_2^2+dx_3^2$ where $p:=p(x_3)$ satisfies $(p')^2=\beta p^{-1}+\gamma$ for constants $\beta<0$ and $\gamma$. The
  potential function is $f=c_1p-1$ where $c_1(x_1)$ satisfies $c_1''+\gamma c_1=0$.

\end{thm}

The spaces in  {\rm  (iii)} of Theorem \ref{th3} are complete and their scalar curvature equals $0$, whereas most of the spaces in {\rm  (ii)} are incomplete.
We presume that even local spaces of Theorem \ref{th3} may have geometric importance, as discussed above.

\bigskip
The gradient vector field $\nabla f $ plays an important role in the study of gradient Ricci solitons, $V$-static spaces, and critical point metrics.  If one can show that $\nabla f$ is a Ricci-eigen vector, then the geometric equation becomes quite tractable. For most three-dimensional Ricci-degenerate spaces satisfying (\ref{ggrsdef}), $\nabla f$ is not Ricci-eigen. This fact requires careful and elaborate arguments when using  $\nabla f$.

Hence, to prove theorems, we consider orthonormal frame fields $\{E_i\ | \ i=1,2,3\}$ such that $\lambda_1=R(E_1,E_1)\neq\lambda_2=R(E_2,E_2)=\lambda_3$ and $E_2\perp\nabla f$. With regard to $E_3$, we have two possible cases: (1) $g(E_3,\nabla f)=E_3f=0$ so that $\nabla f$ is Ricci-eigen, and (2) $E_3f\neq0$.  When $E_3f=0$, we show that $g$ is a warped product metric through well-known arguments.

 When $E_3f\neq0$, it can be shown that there is
a local coordinate system $(x_1,x_2,x_3)$ for each $p\in\{\nabla f\neq0\}$  such that the metric $g$ can be written as $g=g_{11}dx_1^2+g_{22}dx_2^2+g_{33}dx_3^2$, where $E_i=\frac{\partial_i}{\sqrt{g_{ii}}}$. Further, we get a more concrete form of the metric on the basis of the fact that $\lambda_2=\lambda_3$ as well as fundamental properties such as the Jacobi identity.
Next, we explore each of the three classes of spaces mentioned above.  We observe that there exists a Codazzi tensor whose eigenspaces coincide with those of the Ricci tensor. From the presence of this Codazzi tensor, we can show that the $\lambda_2$-eigenspace forms an integrable and umbilic distribution. This provides additional geometric information and a sufficiently good form of the metric $g$ to make conclusive arguments.

\bigskip
The remainder of this paper is organized as follows. In Section 2, we discuss some basic properties of Riemannian manifolds satisfying geometric equations. In Section 3, we classify three-dimensional gradient Ricci solitons with $E_3f\neq0$.
 In Section 4, we classify $V$-static spaces and critical point metrics with $E_3f\neq0$. Finally, in Section 5, we prove our main theorems by dealing with the case where $E_3f=0$.

\section{Three-dimensional manifolds satisfying geometric equation}
In this section, we present our notations and discuss already known basic properties of a three-dimensional manifold that admits a smooth function $f$ satisfying (\ref{ggrsdef}).

Our notational convention is as follows: for orthonormal vector fields $E_i$, $i=1, \cdots, n$ on an $n$-dimensional Riemannian manifold, the curvature components are
$R_{ijkl}:=R(E_i, E_j, E_k, E_l) = < \nabla_{E_i} \nabla_{E_j} E_k - \nabla_{E_j} \nabla_{E_i} E_k  -  \nabla_{[E_i, E_j]} E_k , E_l>  $.
If $F$ is a function of one variable, i.e., $F=F(x)$, then we denote $\frac{\partial}{\partial x}F$ by $F'$.

As is already well known, the Weyl tensor vanishes in a three-dimensional manifold.
 Therefore, (\ref{ggrsdef}) gives us more geometric information than higher-dimensional $n\geq 4$ cases.

\begin{lemma} \label{basiccodazzi}
  Let $(M^3,g,f)$ be a three-dimensional Riemannian manifold satisfying (\ref{ggrsdef}). Then,

\rm{  (i)} $2(d\psi)R+\psi(dR)+4d\phi=2Rc(\nabla \psi-\nabla f,\cdot).$

\rm {(ii)} For a tensor $A_{ijk}=(R_{jk}-\frac{R}{2}g_{jk})(\nabla_i\psi+\nabla_if)+\frac{1}{2}R_{il}(\nabla_l\psi+\nabla_lf)g_{jk}+\psi(\nabla_iR_{jk}-\frac{\nabla_iR}{4}g_{jk})$, $A_{ijk}=A_{jik}$.
\end{lemma}
\begin{proof}
  We take the divergence of (\ref{ggrsdef}):
  \begin{align*}
    \psi \cdot div Ric_b=&\nabla^i\nabla_i\nabla_bf-(\nabla^i\psi)R_{ib}-(\nabla^i\phi)g_{ib}\\
    =&\nabla_b\Delta f+(\nabla^if)R_{ib}-(\nabla^i\psi)R_{ib}-\nabla_b\phi.
  \end{align*}
By $\Delta f=\psi R+3\phi$ and Schur's lemma, $\nabla R=2divRic$, we obtain (i).

By the Ricci identity,
\begin{align*}
  -R_{ijkl}\nabla_l f=&\nabla_i\nabla_j\nabla_kf-\nabla_j\nabla_i\nabla_kf\\
  =&(\nabla_i\psi)R_{jk}-(\nabla_j\psi)R_{ik}+\psi(\nabla_iR_{jk}-\nabla_jR_{ik})+(\nabla_i\phi)g_{jk}-(\nabla_j\phi)g_{ik}.
\end{align*}

Furthermore, in the third dimension, $R_{ijkl}\nabla_lf=-(R_{ik}\nabla_jf+R_{jl}g_{ik}\nabla_lf-R_{il}g_{jk}\nabla_lf-R_{jk}\nabla_if)+\frac{R}{2}(g_{ik}\nabla_jf-g_{jk}\nabla_if)$.
Thus, we get
\begin{align*}
  &R_{jk}(\nabla_i\psi+\nabla_if)+(\nabla_i\phi)g_{jk}+R_{il}\nabla_lfg_{jk}+\psi\nabla_iR_{jk}-\frac{R}{2}\nabla_if g_{jk}\\
  =&R_{ik}(\nabla_j\psi+\nabla_jf)+(\nabla_j\phi)g_{ik}+R_{jl}\nabla_lfg_{ik}+\psi\nabla_jR_{ik}-\frac{R}{2}\nabla_jf g_{ik}.
\end{align*}
By substituting $\nabla_i\phi=-\frac{\psi}{4}\nabla_iR-\frac{\nabla_i\psi}{2}R+\frac{1}{2}R_{il}(\nabla_l\psi-\nabla_lf)$ obtained from (i) in the
above equation, we can get (ii).
\end{proof}

As mentioned in the Introduction, we consider three-dimensional manifolds with two distinct Ricci eigenvalues, say $\lambda_1\neq\lambda_2=\lambda_3$.
We can choose an orthonormal Ricci-eigen frame field $\{E_i\}$ in a neighborhood of each point in $\{\nabla f\neq 0\}$ such that $\lambda_1=R(E_1,E_1)\neq \lambda_2=\lambda_3$.
Without loss of generality, we assume that $E_2\perp \nabla f$, i.e., $E_2f=0$. We shall say that such an orthonormal Ricci-eigen frame field $\{E_i\}$ is {\em adapted}.

\medskip

Let $\{E_i\}$ be an adapted frame field. Since $\psi$ and $\phi$ are functions of $f$, $E_2\psi$ and $E_2\phi$ also vanish; then,  so does $E_2R$ by (i) in Lemma \ref{basiccodazzi}. For orthonormal frame fields $\{E_i\}$, we get the following by (ii) in Lemma \ref{basiccodazzi}:
\begin{align}
  &\left(\lambda_i+\frac{1}{2}\lambda_j-\frac{R}{2}\right)(\nabla_j\psi+\nabla_jf)+\psi\left(\nabla_jR_{ii}-\nabla_iR_{ji}-\frac{\nabla_jR}{4}\right)=0 \label{ijcoda}\\
  &\nabla_iR_{jk}=\nabla_jR_{ik} \label{ijkcoda}
\end{align}
for different $i,j,k$.

\begin{lemma}\label{nabla}
Let $(M^3,g,f)$ be a three-dimensional Riemannian manifold satisfying (\ref{ggrsdef}) with Ricci eigenvalues $\lambda_1\neq\lambda_2=\lambda_3$. Consider an adapted frame field $\{E_i\}$ in an open subset of $\{\nabla f\neq 0\}$. Suppose $E_3f\neq0$. Then we get the following:
  \begin{align*}
 & \nabla_{E_1}E_1=\Gamma_{11}^3E_3,\quad \nabla_{E_1}E_2=0,\quad \nabla_{E_1}E_3=-\Gamma_{11}^3E_1\\
 & \nabla_{E_2}E_1=\frac{H}{2}E_2,\quad \nabla_{E_2}E_2=-\frac{H}{2}E_1+\Gamma_{22}^3E_3,\quad \nabla_{E_2}E_3=-\Gamma_{22}^3E_2\\
 & \nabla_{E_3}E_1=\frac{H}{2}E_3,\quad \nabla_{E_3}E_2=0,\quad \nabla_{E_3}E_3=-\frac{H}{2}E_1.
\end{align*}
for functions $\Gamma_{ij}^k$ and $H$.
 And $R_{1221}=-\frac{E_1H}{2}-\Gamma_{11}^3\Gamma_{22}^3-\frac{H^2}{4}$, $R_{1331}=-\frac{E_1H}{2}+E_3\Gamma_{11}^3-(\Gamma_{11}^3)^2-\frac{H^2}{4}$, $ R_{2332}=E_3\Gamma_{22}^3-(\Gamma_{22}^3)^2-\frac{H^2}{4}$.
\end{lemma}
\begin{proof}
Let $\nabla_{E_i}E_j:=\Gamma_{ij}^kE_k$. Using (\ref{ijcoda}) and (\ref{ijkcoda}):
\begin{itemize}
  \item When $i=3,j=2$, we get $E_2\lambda_3=0$ that means  also $E_2\lambda_1=0.$
  \item When $i=1,j=2$, we get $E_2\lambda_1=\Gamma_{11}^2(\lambda_1-\lambda_2)$, so $\Gamma_{11}^2=0$.
  \item  Comparing  the case of $i=3,j=1$ with $i=2,j=1$, we can see that $\Gamma_{21}^2=\Gamma_{31}^3$(denote by $\frac{H}{2}$).
  \item  When $i=1,j=2,k=3$, we get $\Gamma_{21}^3=0$.
  \item When $i=1,j=3,k=2$, we get $\Gamma_{31}^2=0$.
\end{itemize}
 From $\nabla df(E_1,E_2)=\nabla df(E_3,E_2)=0$, we get $\Gamma_{12}^3=\Gamma_{32}^3=0$.  The
 curvature components can be computed directly.
\end{proof}

Several facts can be easily established from the above lemma. First, by assuming that $\lambda_2=\lambda_3$, $R_{1221}$ and $R_{1331}$ must
be equal to each other. Thus, we have
\begin{align}
  E_3\Gamma_{11}^3=\Gamma_{11}^3(\Gamma_{11}^3-\Gamma_{22}^3). \label{r12211331}
\end{align}
Second, from the Jacobi identity of the Lie bracket, we get
\begin{align}
  E_2\Gamma_{11}^3=0,\quad E_1\Gamma_{22}^3+\frac{H}{2}(\Gamma_{22}^3-\Gamma_{11}^3)=0,\quad E_2H=0.
\end{align}

Third,
\begin{align*}
  E_2E_1f=(E_1E_2-[E_1,E_2])f=0\\
  E_2E_3f=(E_3E_2-[E_3,E_2])f=0.
\end{align*}
Then, by $\nabla df(E_2,E_2)=\frac{H}{2}(E_1f)-\Gamma_{22}^3(E_3f)=\psi \lambda_2+\phi$, we have $E_2\Gamma_{22}^3=0$.

Finally, let $E_{ij}$ be the distribution generated by $E_i$ and $E_j$. Then, $E_{ij}$ are integrable for $i,j=1,2,3$. Thus, there exists locally a
coordinate system $(x_1,x_2,x_3)$ such that the metric $g$ can be written as
\begin{align}
  g=g_{11}dx_1^2+g_{22}dx_2^2+g_{33}dx_3^2,\label{startmetric}
\end{align}
where $E_i=\frac{1}{\sqrt{g_{ii}}}\partial_i$ and $g_{ii}$ are functions of $x_j$ for $j=1,2,3$ \cite{JJ}. In this coordinate system,
\begin{align}
  \Gamma_{11}^3=-\frac{E_3g_{11}}{2g_{11}},\quad \Gamma_{22}^3=-\frac{E_3g_{22}}{2g_{22}},\quad H=\frac{E_1g_{33}}{g_{33}}=\frac{E_1g_{22}}{g_{22}}. \label{gammacord}
\end{align}

\begin{lemma}\label{commonmetric}
  Under the hypothesis of Lemma \ref{nabla}, there exists locally a coordinate system $(x_1,x_2,x_3)$ in which
 $<\nabla_{E_2}E_2,E_1>=<\nabla_{E_3}E_3,E_1>=-\frac{H}{2}$ depends only on $x_1$;
 \end{lemma}
\begin{proof}
  We start from the metric in (\ref{startmetric}).
First, we show that the leaves of $E_{23}$ are totally umbilic. We denote $\nabla_{\partial_i}\partial_j=:\gamma_{ij}^k\partial_k$. In this coordinate,
\begin{align*}
  <\nabla_{\partial_i}\partial_j,E_1>=\sqrt{g_{11}}\gamma_{ij}^1=-\frac{\partial_1g_{ij}}{2\sqrt{g_{11}}}
\end{align*}
for $i,j=2,3$. Since $g_{ij}=0$ for $i\neq j$ and (\ref{gammacord}), we can say that $<\nabla_{\partial_i}\partial_j,E_1>=-\frac{H}{2}g_{ij}$.
Thus, $E_{23}$ is totally umbilic. Further, we can see that $E_3H=0$ by taking the trace of the Codazzi-Mainardi equation \cite{CMM}.
\end{proof}

Now we prove Theorem \ref{metricthm1}.

 \begin{pf1}
    Computing $\Gamma_{11}^2=<\nabla_{E_1}E_1,E_2>=0$ in the coordinates in (\ref{startmetric}) gives $\Gamma_{11}^2=-\frac{\partial_2g_{11}}{2g_{11}\sqrt{g_{22}}}=0$.
  Thus, we get $\partial_2g_{11}=0$. Similarly, by calculating $\Gamma_{32}^3=0$, we get $\partial_2g_{33}=0$. By (\ref{gammacord}), $\frac{\partial_1g_{22}}{g_{22}}=\frac{\partial_1g_{33}}{g_{33}}$.
  Thus, $g_{22}=k(x_2,x_3)g_{33}$ for a function $k(x_2,x_3)$. However, note that $\partial_3\{\ln k(x_2,x_3)\}=\frac{\partial_3g_{22}}{g_{22}}-\partial_3(\ln g_{33})=-2\Gamma_{22}^3\sqrt{g_{33}}-\partial_3(\ln g_{33}).$
  Then, $\partial_2\partial_3(\ln k)=0$, which means that $k(x_2,x_3)=q(x_2)v(x_3)$ for functions $v$ and $q$. By replacing $x_2$ with a new variable, which we still denote by $x_2$, we can replace $q(x_2)dx_2^2$ with $dx_2^2$. Then we get the result.
 \end{pf1}

  \begin{lemma}\label{23second}
    Under the hypothesis of Lemma \ref{nabla}, the following statements hold;

   {\rm (i)} In the local coordinate system in Theorem \ref{metricthm1}, $\partial_3f=c_1(x_1)g_{33}\sqrt{v}$ for a  function $c_1\neq0$;

   {\rm (ii)} For the metric $g$ in Theorem \ref{metricthm1}, if $\Gamma_{11}^3=0$, then $g_{11}=1$, $g_{33}=e^{\int_c^{x_1}H(u)du}$. Otherwise, $\partial_3g_{11}=c_2(x_1)g_{33}\sqrt{g_{11}\cdot v}$ for a function $c_2\neq 0$.

  \end{lemma}

\begin{proof}

Consider the local coordinate system in Theorem \ref{metricthm1}.
From $\nabla df(E_2,E_2)=\nabla df(E_3,E_3)$, we have $E_3E_3f+\Gamma_{22}^3(E_3f)=0$. Since $E_3f\neq 0$, by (\ref{gammacord}), $\frac{\partial_3(E_3f)}{E_3f}-\frac{\partial_3g_{22}}{2g_{22}}=0$; then, we get
$\partial_3f=c_1(x_1)\sqrt{g_{22}g_{33}}=c_1(x_1)g_{33}\sqrt{v}$ for a function $c_1$ by integration.

 Suppose that $\Gamma_{11}^3=0$. By (\ref{gammacord}), $\partial_3g_{11}=0$; thus, $g_{11}=g_{11}(x_1)$. By replacing $g_{11}(x_1)dx_1^2$ with $dx_1^2$, we may write $g=dx_1^2+g_{33}vdx_2^2+g_{33}dx_3^2$.
 Then, by (\ref{gammacord}), $H(x_1)=\partial_1(\ln g_{33})$. Thus, $g_{33}=k(x_3)e^{\int_c^{x_1}H(u)du}$ for a function $k$. By changing the variable, the metric $g$ can be written as
 \begin{align*}
  g= dx_1^2+e^{\int_c^{x_1}H(u)du}vdx_2^2+e^{\int_c^{x_1}H(u)du}dx_3^2.
 \end{align*}
 Now, suppose that $\Gamma_{11}^3\neq0$. By (\ref{r12211331}) and (\ref{gammacord}), $\partial_3(\ln \Gamma_{11}^3)=-\frac{1}{2}\partial_3(\ln g_{11})+\frac{1}{2}\partial_3(\ln g_{22})$. Thus, $\Gamma_{11}^3=\sqrt{\frac{g_{33}\cdot v}{g_{11}}}\cdot c(x_1)$ for a nonzero function $c$. Then, we can get $\partial_3g_{11}=c_2(x_1)g_{33}\sqrt{g_{11}\cdot v}$ for a nonzero function $c_2$ again by (\ref{gammacord}).
\end{proof}

To analyze the metric obtained from the above lemma in greater detail, we need a Codazzi tensor, as mentioned in the Introduction.

\begin{lemma}\label{existcodazzi}
   Let $(M^3,g,f)$ be a three-dimensional Riemannian manifold satisfying (\ref{ggrsdef}). Then $C:=a Rc+b g$ is a Codazzi tensor if and only if two functions $a$ and $b$ satisfy the following conditions:
   \begin{align}
     &\frac{\nabla_i a}{a}=\frac{\nabla_i\psi+\nabla_if}{\psi}\label{acond}\\
     &\frac{b_i}{a}=\frac{1}{\psi}\left\{\frac{1}{2}R_{il}(\nabla_l\psi+\nabla_lf)-\frac{R}{2}(\nabla_i\psi+\nabla_if)-\frac{\nabla_iR}{4}\psi\right\}\label{bcond}
   \end{align}
\end{lemma}
\begin{proof}
  Let $C:=a Rc+bg$ be a Codazzi tensor. Then $0=\nabla_iC_{jk}-\nabla_jC_{ik}=a(\nabla_iR_{jk}-\nabla_jR_{ik})+(\nabla_ia)R_{jk}+(\nabla_ib)R_{jk}-(\nabla_ja)R_{ik}-(\nabla_jb)g_{ik}$. So we have
  \begin{align*}
    \nabla_iR_{jk}-\nabla_jR_{ik}=\frac{\nabla_ja}{a}R_{ik}-\frac{\nabla_ia}{a}R_{jk}+\frac{\nabla_jb}{a}g_{ik}-\frac{\nabla_ib}{a}g_{jk}.
  \end{align*}
  By (ii) in Lemma \ref{basiccodazzi},
  \begin{align*}
    \psi(\nabla_iR_{jk}-\nabla_jR_{ik})=(R_{ik}-\frac{R}{2}g_{ik})(\nabla_j\psi+\nabla_jf)-(R_{jk}-\frac{R}{2}g_{jk})(\nabla_i\psi+\nabla_if)\\
    +\frac{1}{2}R_{jl}(\nabla_l\psi+\nabla_lf)g_{ik}
    -\frac{1}{2}R_{il}(\nabla_l\psi+\nabla_lf)g_{jk}-\psi(\frac{\nabla_jR}{4}g_{ik}-\frac{\nabla_iR}{4}g_{jk}).
  \end{align*}
  Therefore, by comparing the coefficients of $Rc$ and $g$ in the above two equations, we get the desired result.

   For the converse part, one can easily check that if $a$ and $b$ satisfy (\ref{acond}) and (\ref{bcond}), then $aRc+bg$ is a Codazzi tensor.
\end{proof}

In general, it is not easy to obtain $C$ explicitly. However, in subsequent sections,
we will show that if $(M^3,g,f)$ is one of the manifolds listed in the Introduction, then $C$ can be obtained explicitly.

\begin{lemma}\cite{De} \label{derdlem}
For a Codazzi tensor $C$ on a Riemannian manifold $M$,
in each connected component of $M_C:=\{x\in M|$ the number of distinct eigenvalues of $C_x$ is constant in a nbd of $x$ $\}$.

{\rm (i)} $M_C$ is an open dense subset of $M$ and that in each connected component of $M_C$

{\rm (ii)} Given distinct eigenfunctions $\lambda, \mu$ of $C$ and local vector fields $v, u$ such that  $C v = \lambda v$, $Cu = \mu u$ with $|u|=1$, it holds that

$ \ \ \ \ \  v(\mu) = (\mu - \lambda) <\nabla_u u, v > $.

{\rm (iii)} For each eigenfunction $\lambda$, the $\lambda$-eigenspace distribution is integrable and its leaves are totally umbilic submanifolds of $M$.

{\rm (iv)} If $\lambda$-eigenspace $V_{\lambda}$ has dimension bigger than one, the eigenfunction $\lambda$ is constant along the leaves of $V_{\lambda}$.

{\rm (v)} Eigenspaces of $C$ form mutually orthogonal differentiable distributions.
\end{lemma}

In Lemma \ref{derdlem}, (ii) states that it holds not only for distinct eigenfunctions but also for the same eigenfunction as long as the eigenvectors are orthogonal, i.e., $Cv=\lambda v,Cu=\lambda u,u\perp v$.

In general, a Riemannian manifold satisfying (\ref{ggrsdef}) is not real analytic, but if $(M,g)$ is one of the spaces stated in the Introduction, i.e., gradient soliton, $V$-static, or critical point metric, then we can show that $(M,g)$ is real analytic in harmonic coordinates; see \cite{HPW} or \cite{CM2}. Thus, if $f$ is not a constant, then $M_C\cap \{\nabla f\neq0\}$ is open and dense in $M$.

\begin{lemma}\label{giftlem}
  Let $(M^3,g,f)$ be a three-dimensional real analytic Riemannian manifold satisfying (\ref{ggrsdef}) with $\lambda_1\neq \lambda_2=\lambda_3$. Consider an adapted frame field $\{E_i\}$ in an open subset of $\{\nabla f\neq0\}$.
  Suppose that $E_3f\neq0$ and there exists a Codazzi tensor $\mathcal{C}$ whose eigenspaces coincide with  the eigenspace of the Ricci tensor. Let $\mu_i$ be an eigenvalue of $\mathcal{C}$ for $i=1,2,3$. Then, either $\mu_2$ is a constant or $\sqrt{g_{11}}(\mu_2-\mu_1)=c_3(x_1)$ for a positive function $c_3$.
 \end{lemma}
\begin{proof}
First, note that $E_3\mu_2=0$ from (ii) in the above lemma.
 Lemma \ref{nabla} and Lemma \ref{derdlem}  give the following:
\begin{eqnarray} \label{e1f0}
  -\frac{H}{2}=<\nabla_{E_2}E_2,E_1>=\frac{E_1\mu_2}{\mu_2-\mu_1}.
\end{eqnarray}
Thus,
\begin{align}
  E_j\left(\frac{E_1\mu_2}{\mu_2-\mu_1}\right)=\frac{E_jE_1\mu_2}{\mu_2-\mu_1}+\frac{(E_1\mu_2)(E_j\mu_1)}{(\mu_2-\mu_1)^2}=0 \label{noneigenst1}
\end{align}
for $j=2,3$. However, note that
\begin{align*}
  E_jE_1\mu_2=(\nabla_{E_j}E_1-\nabla_{E_1}E_j+E_1E_j)\mu_2=-(\nabla_{E_1}E_j)\mu_2.
\end{align*}
The second equality is due to the fact that $E_j\mu_2=0$ for $j=2,3$. Then, (\ref{noneigenst1}) can be expressed as
\begin{align}
    (E_1\mu_2)\{E_j\mu_1+\Gamma_{11}^j(\mu_2-\mu_1)\}=0. \label{codazzigift}
  \end{align}
 Since $E_3\mu_2=0$, if $E_1\mu_2=0$, then $\mu_2$ is a constant.
  Otherwise,
  \begin{align*}
    0=E_3\mu_1+\Gamma_{11}^3(\mu_2-\mu_1)=\frac{\partial_3\mu_1}{\sqrt{g_{33}}}-\frac{\partial_3g_{11}}{2g_{11}\sqrt{g_{33}}}(\mu_2-\mu_1).
  \end{align*}
 By integrating, we get $\sqrt{g_{11}}(\mu_2-\mu_1)=c_3(x_1)$ for a positive function $c_3$.
\end{proof}

\section{Gradient Ricci soliton with $E_3f\neq0$}

In this section, we consider gradient Ricci solitons with $E_3f\neq0$, i.e., $\psi=-1$ and $\phi=\lambda$ in (\ref{ggrsdef}). First, we recall
some basic properties of gradient Ricci solitons. Then, we observe that there exists a Codazzi tensor.

 \begin{lemma} \label{solitonformulas}
For any gradient Ricci soliton $(M,g,f)$, we have;

\smallskip
{\rm (i)} $\frac{1}{2} dR = R(\nabla f, \cdot ) $, where $R$ in the left hand side denotes the scalar curvature, and $R(\cdot, \cdot)$ is a Ricci tensor.

\smallskip
{\rm (ii)} $R + |\nabla f|^2 - 2\lambda f = constant$.
 \end{lemma}

\begin{lemma}
  Let $(M^3,g,f)$ be a three-dimensional gradient Ricci soliton. Then
  \begin{align}
    \mathcal{T}=e^{-f}(Rc-\frac{R}{2}g)
  \end{align}
 is a Codazzi tensor.
\end{lemma}
\begin{proof}
This is already shown in \cite{CMM}. One can easily check that $a=e^{-f}$ and $b=-e^{-f}\frac{R}{2}$ satisfy the conditions in Lemma \ref{existcodazzi}.
\end{proof}

As mentioned earlier, the presence of this Codazzi tensor immediately gives additional information about the geometry of $(M,g,f)$.
By Lemma \ref{derdlem},
  \begin{align*}
    0=E_3\mu_2=E_3\left\{e^{-f}\left(\lambda_2-\frac{R}{2}\right)\right\}=E_3\left\{e^{-f}\left(-\frac{\lambda_1}{2}\right)\right\}.
  \end{align*}
  Thus, $E_3\lambda_1=\lambda_1(E_3f)$. In the local coordinate system $(x_1,x_2,x_3)$ defined in Theorem \ref{metricthm1}, it is equivalent to
  \begin{align}
    \lambda_1=c_4(x_1)e^f \label{solitonlambda1}
  \end{align}
  for a positive function $c_4$ by integration.
  Further, since $E_3R=2\lambda_2(E_3f)$ by Lemma \ref{solitonformulas},
\begin{align}
  E_3\lambda_2=\frac{1}{2}(E_3R-E_3\lambda_1)=\left(\lambda_2-\frac{\lambda_1}{2}\right)(E_3f). \label{sole3lambda2}
\end{align}

Let us first consider the case where $\mu_2$ is not a constant.
\begin{lemma}\label{gamma113}
  If $\mu_2$ is not a constant, then we get $E_3f=0$, which is a contradiction.
\end{lemma}
\begin{proof}
  Assuming that $\mu_2$ is not a constant, we have $e^{-f}(\lambda_2-\lambda_1)=\mu_2-\mu_1=\frac{c_3(x_1)}{\sqrt{g_{11}}}$ from Lemma \ref{giftlem}. Thus, $\lambda_2=\left(c_4+\frac{c_3}{\sqrt{g_{11}}}\right)e^{f}$ by (\ref{solitonlambda1}). Suppose that $\Gamma_{11}^3\neq0$.  Taking the $\partial_3$-derivative,
  \begin{align}
    \partial_3\lambda_2=&c_3(g_{11})^{-\frac{3}{2}}(\partial_3g_{11})e^f+\left(c_4+\frac{c_3}{\sqrt{g_{11}}}\right)e^f(\partial_3f).
  \end{align}
  However, $\partial_3\lambda_2=\left(\lambda_2-\frac{\lambda_1}{2}\right)\partial_3f$ by (\ref{sole3lambda2}). By the result $\partial_3f=c_1g_{33}\sqrt{v}$ in Lemma \ref{23second},
 \begin{align*}
   c_3(g_{11})^{-\frac{3}{2}}c_2g_{33}\sqrt{g_{11}v}e^f+\left(c_4+\frac{c_3}{\sqrt{g_{11}}}\right)e^fc_1g_{33}\sqrt{v} \\
   =\left\{\left(c_4+\frac{c_3}{\sqrt{g_{11}}}\right)e^f-\frac{c_4}{2}e^f\right\}c_1g_{33}\sqrt{v}.
 \end{align*}
 Thus, we get $\frac{c_1c_4}{2}+\frac{c_3c_2}{g_{11}}=0$. Therefore, if $c_2c_3\neq 0$, then $g_{11}$ is a function of $x_1$ only, which means that $\Gamma_{11}^3=-\frac{\partial_3g_{11}}{2g_{11}\sqrt{g_{33}}}=0$.
 Thus, $c_2c_3$ must be zero. However, we are then in a contradictory situation where $\mu_1=\mu_2$ or $\Gamma_{11}^3=0$. Therefore, we can see that our assumption that $\Gamma_{11}^3$ is not zero is incorrect.

Now, suppose that $\Gamma_{11}^3=0$. By Lemma \ref{23second}, there exists locally a coordinate system $(x_1,x_2,x_3)$ in which
\begin{align}
  g=dx_1^2+e^{\int_c^{x_1}Hdu}v(x_3)dx_2^2+e^{\int_c^{x_1}Hdu}dx_3^2. \label{113vanish}
\end{align}
By Lemma \ref{nabla},
\begin{align*}
  c_4(x_1)e^f=\lambda_1=2R_{1221}=-H'(x_1)-\frac{H(x_1)^2}{2}.
\end{align*}
Hence, either $f$ is a function of $x_1$ only or $c_4=0=-H'-\frac{H^2}{2}$, which implies that $\lambda_1=0$. If $\lambda_1=0$, then $\mu_2=e^{-f}(\lambda_2-\frac{R}{2})=e^{-f}
(-\frac{\lambda_1}{2})=0$. However, we assumed that $\mu_2$ is not a constant. Hence, $f$ should depend only on $x_1$.
\end{proof}

Now, we suppose that $\mu_2$ is a constant. Since $\mu_2=e^{-f}(\lambda_2-\frac{R}{2})=e^{-f}(-\frac{\lambda_1}{2})$, we have $\lambda_1=a e^f$ for a constant $a$. From $\frac{E_1\mu_2}
{\mu_2-\mu_1}=-\frac{H}{2}=-\frac{E_1g_{33}}{2g_{33}}$, we obtain $H=0$ and $\partial_1g_{33}=0$. Hence, the metric $g$ can be written as
\begin{align}
  g=g_{11}(x_1,x_3)dx_1^2+v(x_3)dx_2^2+dx_3^2.
\end{align}

\begin{lemma} \label{solitonresult}
 Let $(M,g,f)$ be a three-dimensional gradient Ricci soliton with $\lambda_1\neq\lambda_2=\lambda_3$. Suppose that $E_3f\neq0$ and  $\mu_2$ is
  a constant. Then, there exists locally a coordinate system $(x_1,x_2,x_3)$ such that the metric $g$ is
 \begin{align}
   g=dx_1^2+k'(x_3)^2dx_2^2+dx_3^2, \label{solmetricresult}
 \end{align}
where $k$ is a solution of
\begin{align}
  k''-\frac{(k')^2}{2}+\lambda k=C
\end{align}
for a constant $C$. Further, the potential function $f=\frac{\lambda}{2}x_1^2+k(x_3)$. In particular, $(\tilde{g},\tilde{f})=(k'(x_3)^2dx_2^2+dx_3^2,k)$ is a two-dimensional gradient
Ricci soliton. Conversely, any metrics $g$ and $f$ in the above form satisfy (\ref{soldef}).
\end{lemma}
\begin{proof}
Since $H=0$, the curvature components are as follows:
\begin{align*}
  R_{1221}=-\Gamma_{11}^3\Gamma_{22}^3,\quad R_{1331}=E_3\Gamma_{11}^3-(\Gamma_{11}^3)^2,\quad R_{2332}=E_3\Gamma_{22}^3-(\Gamma_{22}^3)^2.
\end{align*}
From $\nabla df(E_2,E_2)=\lambda-\lambda_2$,
\begin{align}
  -\Gamma_{22}^3(E_3f)=\lambda -(E_3\Gamma_{22}^3)+(\Gamma_{22}^3)^2+\Gamma_{11}^3\Gamma_{22}^3. \label{soso1}
\end{align}
Suppose that $\lambda_1\neq 0$. Taking the $E_3$-derivative of $\lambda_1=-2\Gamma_{11}^3\Gamma_{22}^3=ae^f$, we get $E_3f=\frac{E_3\Gamma_{11}^3}{\Gamma_{11}^3}+\frac{E_3\Gamma_{22}^3}{\Gamma_{22}^2}$. By (\ref{soso1}), $\lambda+(\Gamma_{22}^3)^2+\Gamma_{11}^3\Gamma_{22}^3+\frac{E_3\Gamma_{11}^3}{\Gamma_{11}^3}\Gamma_{22}^3=0$. Then,
$\lambda+2\Gamma_{11}^3\Gamma_{22}^3=0$ by (\ref{r12211331}). This implies that $f$ is a constant function. Hence, we can
conclude that $\lambda_1=0$. However, if $\Gamma_{22}^3=0$, then $\lambda_1=\lambda_2=0$. Hence,
$\Gamma_{11}^3=-\frac{E_3g_{11}}{2g_{11}}$ must be zero, which implies that $g_{11}$ does not depend on $x_3$.  Thus, the metric $g$ can now be written as
\begin{align*}
  g=dx_1^2+v(x_3)dx_2^2+dx_3^2.
\end{align*}
We can easily see that $\lambda_2=-\frac{v''}{2v}$. From $\nabla df=-Rc+\lambda g$,
\begin{align}
  \partial_1\partial_1f&=\lambda, \label{e1e1}\\
 \partial_1\partial_3f&=0,  \label{e1e3}\\
  \frac{v'}{2v}(\partial_3f)&=\frac{v''}{2v}-\frac{1}{4}\left(\frac{v'}{v}\right)^2+\lambda,  \label{e2e22}\\
  \partial_3\partial_3f&=\frac{v''}{2v}-\frac{1}{4}\left(\frac{v'}{v}\right)^2+\lambda.     \label{e3e3}
\end{align}
From (\ref{e1e3}), $\partial_3f$ depends only on $x_3$. Then, if necessary, through coordinate change, we get $v=(\partial_3f)^2$ by (\ref{e2e22}) and (\ref{e3e3}). Similarly, $f=\frac{\lambda}{2}x_1^2+k(x_3)$ by (\ref{e1e1}) and (\ref{e1e3}). Let $p(x_3):=(\partial_3f)(x_3)$. Then, we
get $p'=\frac{p''}{p}+\lambda$ by (\ref{e2e22}). Let us consider (ii) in Lemma \ref{solitonformulas}. Assigning what we have obtained thus far, we get $-2\frac{p''}{p}+(k')^2-2\lambda k=C$
for a constant $C$. Since $k'=p$ and $\frac{p''}{p}=p'-\lambda$, we finally obtain
\begin{align}
  k''-\frac{(k')^2}{2}+\lambda k=C \label{ksol}
\end{align}
for a constant $C$. One can easily check that $(\tilde{g}=k'(x_3)^2dx_2^2+dx_3^2,\tilde{f}=k)$ satisfies $\nabla df+Rc=\lambda g$ as long as $k$ is a solution to (\ref{ksol}) and the converse part.
\end{proof}

\section{$V$-static and critical point metric with $E_3f\neq0$}
In this section, we consider the $V$-static ($\kappa\neq0$) and critical point metric simultaneously. These two spaces can
be regarded as the general one
\begin{align}
  \nabla df=(a_1+f)Rc+\left(a_2-\frac{R}{2}f\right)g \label{generaldef}
\end{align}
for constants $a_1$ and $a_2$. Note that the scalar curvature $R$ is constant in these cases.

\begin{lemma}
  Let $(M,g,f)$ be a three-dimensional Riemannian manifold satisfying (\ref{generaldef}). Then, the $(0,2)$-tensor
  \begin{align}
    D:=(a_1+f)^2Rc+\left\{\frac{1}{2}|\nabla f|^2-(a_2+a_1R)f-\frac{R}{4}f^2\right\}g
  \end{align}
  is a Codazzi tensor.
\end{lemma}
\begin{proof}
 One can easily check that $a=(a_1+f)^2$ and $b=\frac{1}{2}|\nabla f|^2-(a_2+a_1R)f-\frac{R}{4}f^2$ satisfy the conditions in Lemma \ref{existcodazzi}.

\end{proof}
Recall that a coordinate system $(x_1,x_2,x_3)$ in Theorem \ref{metricthm1} in which the metric $g$ is as follows exists locally:
\begin{align}
   g=g_{11}(x_1,x_3)dx_1^2+g_{33}(x_1,x_3)v(x_3)dx_2^2+g_{33}(x_1,x_3)dx_3^2. \label{startmetric1}
\end{align}
\begin{lemma}\label{e3mu2tam}
  Let $(M^3,g,f)$ be a three-dimensional Riemannian manifold satisfying (\ref{generaldef}) with $\lambda_1\neq\lambda_2=\lambda_3$. Consider an adapted frame field $\{E_i\}$ in
  an open subset of $\{\nabla f\neq 0\}$. Suppose that $E_3f\neq0$. Then, there exists locally a coordinate system $(x_1,x_2,x_3)$ such that
    \begin{align*}
    (a_1+f)^3=\frac{c_5(x_1)}{3\lambda_2-R},\quad  (3\lambda_2-R)^2=\frac{p_2(x_3)}{g_{33}^3}
  \end{align*}
  for positive functions $c_5$ and $p_2$.
\end{lemma}
\begin{proof}
  We start with the metric $g$ in (\ref{startmetric1}).  By the property of the Codazzi tensor described in Lemma \ref{derdlem},
  \begin{align*}
    0=E_3\mu_2=&E_3\{(a_1+f)^2\lambda_2+\frac{1}{2}|\nabla f|^2-(a_2+a_1R)f-\frac{R}{4}f^2\}\\
    =&(E_3f)(3\lambda_2-R)(a_1+f)+(a_1+f)^2(E_3\lambda_2).
  \end{align*}
 Thus, $\frac{\partial_3f}{a_1+f}=-\frac{\partial_3\lambda_2}{3\lambda_2-R}$. By integrating, we get $(a_1+f)^3=\frac{c_5(x_1)}{3\lambda_2-R}$ for a positive function $c_5$. Again, by Lemma \ref{derdlem},
 \begin{align*}
   E_1\mu_2=(\mu_2-\mu_1)<\nabla_{E_2}E_2,E_1>=(a_1+f)^2(3\lambda_2-R)(-\frac{H}{2}).
 \end{align*}
However, direct computation gives
\begin{align}
  E_1\mu_2=E_1\{(a_1+f)^2\lambda_2+\frac{1}{2}|\nabla f|^2-(a_2+a_1R)f-\frac{R}{4}f^2\}=(a_1+f)^2(E_1\lambda_2). \label{e1mu2dir}
\end{align}
 Thus, we have $\frac{E_1\lambda_2}{3\lambda_2-R}=-\frac{H}{2}=-\frac{E_1g_{33}}{2g_{33}}$. By integrating, we get $(3\lambda_2-R)^2=\frac{p_2(x_3)}{g_{33}^3}$ for a positive function $p_2$.
 \end{proof}

\begin{lemma}\label{metriclemtam}
  Let $(M^3,g,f)$ be a three-dimensional Riemannian manifold satisfying (\ref{generaldef}) with $\lambda_1\neq\lambda_2=\lambda_3$. Consider an adapted frame field $\{E_i\}$ in an open
  subset of $\{\nabla f\neq 0\}$. Suppose that $E_3f\neq0$ and  $\mu_2$ is not a constant. Then, there exists a local coordinate system in which
  \begin{align}
    g=\frac{1}{\{q(x_3)+b(x_1)\}^2}\{dx_1^2+(q')^2dx_2^2+dx_3^2\} \label{113notzerometric}
  \end{align}
 for nonconstant functions $q$ and $b$ such that $b'=-\frac{H}{2}$. The potential function $f=c_8(q+b)^{-1}-a_1$ for a function $c_8$.
\end{lemma}
\begin{proof}

First, consider $\Gamma_{11}^3=0$. Then, the metric $g$ can be written as $g=dx_1^2+e^{\int_c^{x_1}H(u)du}vdx_2^2+e^{\int_c^{x_1}H(u)du}dx_3^2$ by Lemma \ref{23second}. Thus, $\lambda_1=-H'-\frac{H^2}{2}$ depends only on $x_1$. However, $\lambda_2$ is also a function of $x_1$ only, and this implies that $E_3f=0$ by Lemma \ref{e3mu2tam}. Therefore, $\Gamma_{11}^3$ cannot be identically zero.

Now, we assume that $\Gamma_{11}^3\neq0$.   Since we assumed that $\mu_2$ is not a constant, by Lemma \ref{giftlem}, we have
\begin{align}
  \sqrt{g_{11}}(\mu_2-\mu_1)=\sqrt{g_{11}}(a_1+f)^2(3\lambda_2-R)=c_3(x_1)\label{findff}
\end{align}
for a positive function $c_3$. Then, we can get $g_{11}=p_3(x_3)c_6(x_1)g_{33}$ for positive functions $p_3$ and $c_6$ by Lemma \ref{e3mu2tam}.
By substituting $c_6dx_1^2=dx_1^2$, we can write the metric $g$ as
\begin{align*}
  g=p_3g_{33}dx_1^2+vg_{33}dx_2^2+g_{33}dx_3^2.
\end{align*}
By (\ref{gammacord}), $ H(x_1)=\frac{\partial_1g_{33}}{g_{33}\sqrt{g_{11}}}=\frac{1}{\sqrt{p_3}}g_{33}^{-\frac{3}{2}}\partial_1g_{33}.$
By integrating, we get $g_{33}=\frac{1}{\{\int (-H/2)dx_1+p_4(x_3)\}^2p_3}$ for a function $p_4$. By changing the variable,
\begin{align*}
  g=\frac{1}{\{q(x_3)+b(x_1)\}^2}dx_1^2+\frac{v}{(q+b)^2p_3}dx_2^2+\frac{1}{(q+b)^2}dx_3^2
\end{align*}
for functions $q$ and $b$, where $b'=-\frac{H(x_1)}{2}$. Assuming that $\Gamma_{11}^3\neq 0$, (\ref{r12211331}) gives $\partial_3g_{11}=c_7(x_1)\sqrt{g_{11}g_{22}g_{33}}$ as in the proof of Lemma \ref{23second}.
Thus, we can get $q'=C\sqrt{\frac{v}{p_3}}$ for a constant $C$. Hence, the metric $g$ can be written as in (\ref{113notzerometric}). By (\ref{findff}) and Lemma \ref{e3mu2tam}, $c_3(a_1+f)=\sqrt{g_{11}}(a_1+f)^3(3\lambda_2-R)=\sqrt{g_{11}}c_5$. Thus, the potential function $f=\frac{c_8}{q+b}-a_1$ for a function $c_8(x_1)$. Note that if $q'=0$, then $q$ is a constant, which implies that $\Gamma_{11}^3=\frac{E_3g_{11}}{g_{11}}=0$. Thus, $q$ cannot be a constant. Further, if $b$ is a constant, then $H=0$, which implies that $\mu_2$ is a constant. Hence, $b$ also cannot be a constant.
\end{proof}

The curvature components in this coordinate system are calculated as follows:
\begin{align*}
  &H=-2b',\quad \Gamma_{11}^3=q',\quad \Gamma_{22}^3=q'-\frac{q''}{q'}(q+b),\\
  &R_{1221}=(q+b)(q''+b'')-(q')^2-(b')^2,\\
  &R_{2332}=2(q+b)q''-\frac{q'''}{q'}(q+b)^2-(q')^2-(b')^2.
\end{align*}

By Lemma \ref{e3mu2tam},
\begin{align*}
  \sqrt{p_2}(q+b)^3=\sqrt{p_2}g_{33}^{-\frac{3}{2}}=3\lambda_2-R=c_5(a_1+f)^{-3}=c_5c_8^{-3}(q+b)^3.
\end{align*}
Thus, $\sqrt{p_2}=c_5c_8^{-3}$ is a constant, and the Ricci eigenvalues are
\begin{align*}
  \lambda_2=-m(q+b)^3+\frac{R}{3},\quad  \lambda_1=2m(q+b)^3+\frac{R}{3}\\
\end{align*}
for a constant $m\neq0$.

\begin{lemma}\label{cpeode}
  The functions $b$ and $q$ in Lemma \ref{metriclemtam} are solutions of the following ODE, respectively:
  \begin{align}
     &(q')^2-2mq^3-lq^2+\alpha q+k=0 \label{solq}\\
     &(b')^2-2mb^3+lb^2+\alpha b+\frac{R}{6}-k=0 \label{solb}
  \end{align}
  for constants $m\neq 0,l,\alpha,k$. And $c_8$ satisfies $ b''c_8=b'c_8'+a_2+\frac{a_1R}{2}$. Conversely, any metric $g$ and $f$ of the form in Lemma \ref{metriclemtam} satisfying the above differential equations satisfy (\ref{generaldef}).
\end{lemma}
\begin{proof}
  From $\lambda_1=2R_{1221}$ and $\lambda_2=R_{1221}+R_{2332}=\frac{1}{2}\lambda_1+R_{2332}$, we get
  \begin{align}
   & m(q+b)^3+\frac{R}{6}=(q+b)(q''+b'')-(q')^2-(b')^2. \label{lambda1tam}\\
   & 2m(q+b)^3-\frac{R}{6}=\left\{\frac{q'''}{q'}(q+b)-2q''\right\}(q+b)+(q')^2+(b')^2 \label{lambda2tam}
 \end{align}
Assigning $(E_1,E_1)$ and $(E_2,E_2)$ to (\ref{generaldef}),
\begin{align}
 &2m(q+b)^3-\frac{R}{6}=\left\{\frac{c_8''}{c_8}(q+b)-\frac{c_8'b'}{c_8}-b''-\frac{1}{c_8}\left(a_2+\frac{a_1R}{2}\right)\right\}(q+b)+(q')^2+(b')^2 \label{lambda1tam2}\\
 & m(q+b)^3+\frac{R}{6}=\left\{\frac{b'c_8'}{c_8}+q''+\frac{1}{c_8}\left(a_2+\frac{a_1R}{2}\right)\right\}(q+b)-(q')^2-(b')^2 \label{lambda2tam2}
\end{align}
 By (\ref{lambda1tam}) and (\ref{lambda2tam2}), we get
 \begin{align}
   b''c_8=b'c_8'+a_2+\frac{a_1R}{2}\label{relbc8}
 \end{align}
  By (\ref{lambda2tam}) and (\ref{lambda1tam2}), we get
  \begin{align}
    \left(\frac{q'''}{q'}-\frac{c_8''}{c_8}\right)(q+b)=2q''-b''-\frac{c_8'b'}{c_8}-\frac{1}{c_8}\left(a_2+\frac{a_1R}{2}\right). \label{ddee}
  \end{align}
 By (\ref{relbc8}), $\frac{c_8'b'}{c_8}+\frac{y}{c_8}=b''$ and $\frac{c_8''}{c_8}=\frac{b'''}{b'}$. Thus (\ref{ddee}) becomes
 \begin{align}
   \left(\frac{q'''}{q'}-\frac{b'''}{b'}\right)(q+b)=2(q''-b'') \label{ddee2}
 \end{align}
 Adding (\ref{lambda1tam}) and (\ref{lambda2tam}),
 \begin{align*}
   3m(q+b)^3=&\left(b''-q''+\frac{q'''}{q'}(q+b)\right)(q+b)\\
   =&\frac{q'''}{q'}(q+b)^2-\frac{1}{2}\left(\frac{q'''}{q'}-\frac{b'''}{b'}\right)(q+b)^2\quad \because(\ref{ddee2}).
 \end{align*}
Since $q=q(x_3)$ and $b=b(x_1)$, we get $-\frac{b'''}{b'}+6mb=\frac{q'''}{q'}-6mq=l$ for a constant $l$. By integrating,
\begin{align}
  (q')^2-2mq^3-lq^2+\alpha q+k_1=0,\quad (b')^2-2mb^3+lb^2+\beta b+k_2=0 \label{ddee3}
\end{align}
for constants $\alpha,\beta,k_i$. Putting (\ref{ddee3}) in (\ref{ddee2}) shows that $\alpha=\beta$. Assigning these to (\ref{generaldef}), we get $k_1+k_2=\frac{R}{6}$.

 The converse part
can be easily checked by direct computations.
\end{proof}

\begin{remark}
The metric $g$ in Lemma \ref{metriclemtam} cannot be defined on a compact manifold. For $g$ to be so, $\tilde{g}=\frac{(q')^2}{(q+b)^2}dx_2^2+\frac{1}{(q+b)^2}dx_3^2$ for a given $x_1$ should be a metric on a sphere.
By $dx_3=\frac{dq}{q'}$ and (\ref{solq}), $\tilde{g}$ becomes $  \tilde{g}=\frac{2mq^3+lq^2-\alpha q-k}{(q+b)^2}dx_2^2+\frac{1}{(q+b)^2(2mq^3+lq^2-\alpha q-k)}dq^2$. Let $r:=\frac{1}{(q+b)}$. Then, $\tilde{g}=f(r)dx_2^2+\frac{1}{f(r)}dr^2$, where $f(r)=(l-6mb)+(lb^2-2mb^3+\alpha b-k)r^2
+(6mb^2-2bl-\alpha)r+\frac{2m}{r}$. Let $dt:=\frac{dr}{\sqrt{f(r)}}$. Then, $t=\int \frac{dr}{\sqrt{f(r)}}$ and the metric $\tilde{g}$ becomes
$dt^2+f(t(r))dx_2^2$.  Suppose that $\tilde{g}$ is a metric
on a sphere. This means that there exist two points $a\neq b$ such that $f(a)=f(b)=0$, $f'(a)+f'(b)=0$, and $f'(a),f'(b)\neq0$. However, note that
\begin{align*}
  \frac{df}{dt}=(lb^2-2mb^3+\alpha b-k)2r(\frac{dr}{dt})+(6mb^2-2bl-\alpha)(\frac{dr}{dt})-\frac{2m}{r^2}(\frac{dr}{dt}).
\end{align*}
Since $\frac{dr}{dt}=\sqrt{f(r)}$, if $f(a)=0$, then $f'(a)=0$. Thus, $g$ cannot be a compact metric.

\end{remark}
\medskip

Now, we consider the case that $\mu_2$ is a constant. Suppose that $\mu_2$ is a constant. By $\frac{E_1\mu_2}{\mu_2-\mu_1}=<\nabla_{E_2}E_2,E_1>=- \frac{H}{2}=-\frac{E_1g_{33}}{2g_{33}}=0$. Hence, the metric $g$ in Theorem \ref{metricthm1} can be written as
\begin{align*}
  g=g_{11}(x_1,x_3)dx_1^2+v(x_3)dx_2^2+dx_3^2.
\end{align*}

\begin{lemma}\label{cpe2233}
  Let $(M^3,g,f)$ be a three-dimensional Riemannian manifold satisfying (\ref{generaldef}) with $\lambda_1\neq\lambda_2=\lambda_3$. Suppose that $E_3f\neq0$ and  $\mu_2$ is a constant. Then, $(M^3,g,f)$ must be a critical point metric, and there exists a local coordinate system $(x_1,x_2,x_3)$ such that
  \begin{align}
    g=p^2dx_1^2+(p')^2dx_2^2+dx_3^2,\quad f=c_1p-1,
  \end{align}
where $p(x_3)$ satisfies $(p')^2=\beta p^{-1}+\gamma$ for constants $\beta<0$ and $\gamma$, and $c_1(x_1)$ satisfies $c_1''+\gamma c_1=0$. Conversely, any
metrics $g$ and $f$ in the above form satisfy (\ref{cpedef}).

\end{lemma}
\begin{proof}
If $\Gamma_{11}^3 \neq 0$, then $\partial_3g_{11}=c_2\sqrt{g_{11}v}$ by  Lemma \ref{23second}. By integrating, $g_{11}=c_2^2\left(\int\frac{\sqrt{v}}{2}dx_3+a(x_1)\right)^2$ for
a function $a$. By substituting $c_2^2dx_1^2=dx_1^2$ and $4dx_2^2=dx_2^2$, we may write
\begin{align}
  g=(p(x_3)+a(x_1))^2dx_1^2+(p')^2dx_2^2+dx_3^2 \label{ddee33}
\end{align}
for a function $p(x_3)$. In this coordinate, the curvature components are $R_{1221}=-\frac{p''}{p+a}$ and $R_{2332}=-\frac{p'''}{p'}$.
However, note that $E_1\mu_2=0$ means that $\partial_1\lambda_2=0$ by (\ref{e1mu2dir}). Hence, $a(x_1)$ must be a constant function. Therefore, the metric is
\begin{align*}
  g=p^2dx_1^2+(p')^2dx_2^2+dx_3^2
\end{align*}
and the Ricci curvatures are $\lambda_1=-2\frac{p''}{p}$, $\lambda_2=-\frac{p''}{p}-\frac{p'''}{p'}$. By $\partial_3f=c_1g_{33}\sqrt{v}$ in Lemma \ref{commonmetric},  $f=c_1(p+c_2)$ for a function $c_2$.

Then,
\begin{align*}
  0=\nabla df(E_3,E_1)=E_3E_1f=-\frac{p'}{p^2}(c_1'c_2+c_1c_2').
\end{align*}
Thus, $f=c_1p+\alpha$ for a constant $\alpha$. From $\nabla df(E_1,E_1)$ and $\nabla df(E_2,E_2)$,
\begin{align}
  c_1''+c_1(p')^2=-2p''(a_1+c_1p+\alpha)+a_2p-\frac{R}{2}p(c_1p+\alpha) \label{1111}\\
  c_1p''=(a_1+c_1p+\alpha)(-\frac{p''}{p}-\frac{p'''}{p'})+a_2-\frac{R}{2}(c_1p+\alpha) \label{2222}.
\end{align}

Taking the $\partial_3$-derivative of (\ref{1111}) and using $R=-4\frac{p''}{p}-2\frac{p'''}{p'}$, we get
$c_1p''=\frac{p''}{p}(a_1+c_1p+\alpha)+\frac{R}{4}a_1+\frac{R}{8}\alpha+\frac{a_2}{4}$. Comparing with (\ref{2222}),
\begin{align}
  (a_1+\alpha)p''+(\frac{R}{2}a_1+a_2)p=0\\
  Ra_1+3a_2-\frac{R}{2}\alpha=0.
\end{align}
If $a_1+\alpha\neq0$, then $\frac{p''}{p}=\frac{p'''}{p'}=-\frac{\frac{R}{2}a_1+a_2}{a_1+\alpha}$, which implies that $\lambda_1=\lambda_2$, which is a contradiction. Therefore, $a_1+\alpha=\frac{R}{2}a_1+a_2=0$. This equation cannot be satisfied in  $V$-static spaces but only in critical point spaces under the condition of $R=0$. If $R=0$, then $-2\frac{p''}{p}-\frac{p'''}{p'}=0$; hence, $p$ is a solution of $(p')^2=\beta p^{-1}+\gamma$ for constants $\beta<0$ and $\gamma$. By (\ref{1111}), $c_1$ satisfies $c_1''+\gamma c_1=0$.

Now, suppose that $\Gamma_{11}^3=0$. Then, the metric $g$ can be written as
\begin{align*}
  g=dx_1^2+p(x_3)^2dx_2^2+dx_3^2.
\end{align*}
From $\nabla df(E_3,E_1)=\partial_3\partial_1f=0$, we have $f=f_1(x_1)+f_3(x_3)$ for functions $f_1(x_1)$ and $f_3(x_3)$. On the other hand, $\nabla df(E_1,E_1)=\partial_1\partial_1f=a_2-\frac{R}{2}f$ implies that  $R=0$, because we assumed that $E_3f\neq0$. However, note that if $R=0$, then $\lambda_1=\lambda_2=0$, which is a contradiction.

The converse part can be easily checked.
\end{proof}
\begin{remark}
 The metric $g$ in Lemma \ref{cpe2233} cannot be defined on a compact manifold. For the metric $g$ to be on a compact metric, there should exist two points $a\neq b$ such
 that $p'(a)=p'(b)=0$ and $p''(a)+p''(b)=0$ \cite{Pe}. If $p'(a)=p'(b)=0$, then $0=(p'(a))^2=\beta p(a)^{-1}+\gamma=\beta p(b)^{-1}+\gamma=(p'(b))^2=0$, which implies that $p(a)=p(b)$. However, then $p''(a)+p''(b)=-\frac{\beta}{2}(\frac{1}{p(a)^2}+\frac{1}{p(b)^2})=-\frac{\beta}{p(a)^2}$ cannot be
 zero unless $\beta=0$.

\end{remark}
\begin{remark}
  In \cite{Ko}, Kobayashi studied the behavior of the solution to $k=(r')^2+\frac{2a}{n-2}r^{2-n}+\frac{R}{n(n-1)}r^2$. Our equation in Lemma \ref{cpe2233} corresponds to the case ${\rm (IV.1)}$ in the list of \cite[p 670]{Ko}.
  One can check that the metric $g$ in Lemma \ref{cpe2233} is complete by simple calculus; see \cite[Lemma 4.5]{JJ}.
\end{remark}

\section{Three-dimensional Ricci-degenerate manifolds with $E_3f=0$}

In this section, we deal with the case where $E_3f=0$. Unlike the previous section, this section covers the three spaces simultaneously.  First, consider adapted frame $\{E_i,i=1,2,3\}$. Then $E_3f=0$ implies
that $\nabla f$ is parallel to $E_1$. We may set $\frac{\nabla f}{|\nabla f|}=E_1$. Then we can prove the following lemma by standard argument; see \cite{CMMR} or \cite{Ki}.
\begin{lemma}\label{threesolb1}
  Let $(M^3,g,f)$ be a three-dimensional Ricci-degenerate Riemannian manifold with $\lambda_1\neq\lambda_2=\lambda_3$ satisfying (\ref{ggrsdef}). Let $c$ be a regular value of $f$ and
  $\Sigma_c=\{x|f(x)=c\}$. If $\frac{\nabla f}{|\nabla f|}=E_1$, then the following hold.

{\rm  (i)} $R$ and $|\nabla f|^2$ are constant on a connected component of $\Sigma_c$.

 {\rm (ii)} There is a function $s$ locally defined with $s(x)=\int\frac{df}{|\nabla f|}$, so that $ds=\frac{df}{|\nabla f|}$ and $E_1=\nabla s$.

{\rm  (iii)} $\nabla_{E_1}E_1=0$.

{\rm  (iv)} $\lambda_1$ and $\lambda_2$ are constant on a connected component of $\Sigma_c$ and so depend on the local variable $s$ only.

{\rm (v)}  Near a point in $\Sigma_c$, the metric $g$ can be written as

$\ \ \ g= ds^2 +  \sum_{i,j >  1} g_{ij}(s, x_2, \cdots  x_n) dx_i \otimes dx_j$, where
    $x_2, \cdots  x_n$ is a local coordinates system on $\Sigma_c$.

{\rm (vi)} $\nabla_{E_i}E_1=\zeta(s) E_i, \textrm{$i=2,3$ with }\zeta(s)=\frac{\psi \lambda_i+\phi}{|\nabla f|}$  and
$g(\nabla_{E_i}E_i, E_1)=-\zeta $.

\end{lemma}
\begin{proof}
By assumption, for $i=2,3$,  $R(\nabla f, E_i) =0$ and Lemma \ref{basiccodazzi} (i) gives $E_i(R) =0$. Equation {\rm (\ref{ggrsdef})} gives $E_i(|\nabla f|^2) =0$. We can see $d( \frac{df}{|df |} )=0$.
$g(\nabla_{E_1} E_1 , E_1)=0$ is trivial. We can get
$g(\nabla_{E_1} E_1 , E_i)=g(\nabla_{E_1} (\frac{\nabla f}{|\nabla f |} ), E_i) =0$ from  {\rm (\ref{ggrsdef})}. (i), (ii) and (iii) are proved.
As $\nabla f$ and the level surfaces of $f$ are perpendicular, we get (v).

Assigning $(E_1,E_1)$ to (\ref{ggrsdef}), we have $E_1E_1f=\psi\lambda_1+\phi$. Since $\psi\neq0$, $\lambda_1$ is a function of $s$ only. Then
$\lambda_2$ also depends only on $s$ from the fact that $R=R(s)$.  So we proved {\rm (iv)} and {\rm (vi)}.
\end{proof}



\begin{lemma} \label{claim112b3}
Let  $(M^3,g,f)$ be  a three-dimensional Ricci-degenerate Riemannian manifold satisfying (\ref{ggrsdef}).
Suppose there exists a Codazzi tensor $C$ whose eigenspace coincide with Ricci eigenspace. Suppose that $\frac{\nabla f}{|\nabla f|}=  E_1$ and   $ \lambda_1  \neq \lambda_2=  \lambda_3$ for an adapted frame fields $\{ E_j\}$,
on an open subset $U$ of $\{ \nabla f \neq 0  \}$.

\smallskip
Then for each point $p_0$ in $U$, there exists a neighborhood $V$ of $p_0$ in $U$ with coordinates $(s, x_2, x_3)$  such that $\nabla s= \frac{\nabla f }{ |\nabla f |}$ and $g$ can be written on $V$ as
\begin{equation} \label{mtr1a3}
g= ds^2 +      h(s)^2 \tilde{g},
\end{equation}
 where  $h:=h(s)$ is a smooth function and
 $\tilde{g}$ is (a pull-back of) a Riemannian metric of constant curvature on a $2$-dimensional domain with $x_2, x_3$ coordinates.
In particular,  $g$ is locally conformally flat.


\end{lemma}

 \begin{proof}
 The metric $g$  of Lemma \ref{threesolb1} (v) can be written as
\begin{equation} \label{ggg}
g= ds^2 +    g_{22}dx_2^2 +   g_{23} dx_2 \odot dx_3 + g_{33}dx_3^2,
\end{equation}
where $g_{ij}$ are functions  of $(x_1:=s, \ x_2, \  x_3)$.
one easily gets $E_1 =\frac{\partial }{\partial s} $.
 We write $\partial_{1}:=\frac{\partial }{\partial s}$ and $\partial_{1}:=\frac{\partial }{\partial x_i}, i=2,3$ .

We consider the second fundamental form $\tilde{ h}$ of a leaf for $E_{23}$ with respect to $E_1$;
$\tilde{ h}  ( u , u ) =  -  < \nabla_{u} u ,  E_1>  $. As the leaf is totally umbilic by Lemma  \ref{derdlem} {\rm (ii)}, $\tilde{ h} ( u , u ) =    \eta \cdot g( u , u) $ for some function  $\eta$ and any $u$ tangent to a leaf.
Then, $\tilde{ h} (E_2, E_2 ) = -  < \nabla_{E_2} E_2 ,  E_1> = \eta=   \zeta   $,  which is a function of $s$ only by  Lemma \ref{threesolb1} (vi).

For $i, j \in \{ 2,3 \}$,
\begin{eqnarray*}
\zeta  g_{ij}& =\tilde{ h} ( \partial_i , \partial_{j} )  =   -  < \nabla_{\partial_i}  \partial_{j} ,   \frac{\partial }{\partial s}> =  -  <\sum_k \Gamma^{k}_{i{j}}  \partial_k ,   \frac{\partial }{\partial s}  > \\
& =  - \sum_k <  \frac{1}{2} g^{kl}( \partial_i g_{lj} +\partial_{j} g_{li} - \partial_l g_{ij} )\partial_k ,  \frac{\partial }{\partial s} >       = \frac{1}{2} \frac{\partial }{\partial s} g_{i{j}}.
\end{eqnarray*}
 So,  $\frac{1}{2} \frac{\partial }{\partial s} g_{i{j}} =   \zeta g_{ij}$. Integrating it, for $i, j \in \{ 2,3 \}$, we get $ g_{ij} = e^{C_{ij}} h(s)^2$. Here the function $h(s)>0$ is independent of $i,j$ and each function $C_{ij}$ depends only on $x_2, x_3$.

 Now  $g$ can be  written as $g= ds^2 +     h(s)^2 \tilde{g} $, where $\tilde{g}$ can be viewed as a Rimannian metric in a domain of $(x_2, x_3)$-plane.

From Gauss-Codazzi equation,
$R^{g} = R^{\tilde{g}} + 2 Ric^{g}(E_1,E_1)+ \|h\|^2 - H^2$.  As all others are constant on a hypersurface of $w$, so is $R^{\tilde{g}}$. Therefore each hypersurface has constant curvature. Thus $\tilde{g}$ has constant curvature and
 $g$ is locally conformally flat.
 \end{proof}

Now, we can prove our theorems.
\begin{pf2}
  Combine Lemma \ref{gamma113}, Lemma \ref{solitonresult} and Lemma \ref{claim112b3}.
  \end{pf2}
\begin{pf3}
  Combine Lemma \ref{metriclemtam}, Lemma \ref{cpeode}, Lemma \ref{cpe2233} and Lemma \ref{claim112b3}.
\end{pf3}
\begin{pf4}
  Combine Lemma \ref{metriclemtam}, Lemma \ref{cpeode}, Lemma \ref{cpe2233} and Lemma \ref{claim112b3}.
\end{pf4}

\bigskip
\bigskip

\bigskip
\bigskip

\bigskip
\bigskip

\bigskip
\bigskip

\bigskip
\bigskip

\bigskip
\bigskip

\bigskip
\bigskip

\bigskip
\bigskip

\bigskip
\bigskip

\bigskip
\bigskip

\end{document}